\documentclass[11pt,a4paper]{amsart}
\usepackage[T1]{fontenc}
\usepackage[utf8]{inputenc}
\usepackage{lmodern}
\usepackage[margin=28mm]{geometry}
\usepackage{mathtools,amssymb,microtype}
\usepackage[hidelinks]{hyperref}
\usepackage{xcolor}
\newtheorem{theorem}{Theorem}[section]

\newtheorem{conjecture}[theorem]{Conjecture}

\newtheorem{thmintroduction}{Theorem}
\newtheorem{corintroduction}[thmintroduction]{Corollary}
\newtheorem{conjectureintro}{Conjecture}

\numberwithin{equation}{section}
\DeclareMathOperator{\Ext}{Ext}
\DeclareMathOperator{\Hom}{Hom}
\DeclareMathOperator{\End}{End}
\DeclareMathOperator{\im}{im}
\DeclareMathOperator{\id}{id}
\DeclareMathOperator{\rank}{rank}
\DeclareMathOperator{\tr}{tr}
\newcommand{\HH}{\mathrm H}
\newcommand{\Hoch}{\mathrm{HH}}
\newcommand{\Z}{\mathbb Z}
\newcommand{\Q}{\mathbb Q}
\newcommand{\Fp}{\mathbb F_p}
\title[Nonvanishing of cohomology for groups]{Nonvanishing of cohomology for groups}
\author{Tiago Cruz}
\address{Tiago Cruz, Institut f\"ur Algebra und Zahlentheorie,
Universit\"at Stuttgart, Germany}
\email{tiago.cruz@mathematik.uni-stuttgart.de}
\author{Ren\'e Marczinzik}
\address{Ren\'e Marczinzik, Mathematical Institute of the University of Bonn,
Endenicher Allee 60, 53115 Bonn, Germany}
\email{marczire@math.uni-bonn.de}
\date{\today}

\begin{document}
\begin{abstract}
Let $G$ be a finite group. We prove that the order of the
abelianization $G/G'$ divides the order of $\HH^{2m}(G,\Z)$
for every $m\geq1$. As a consequence, for a field $K$, nonvanishing
of $\Ext^1_{KG}(K,K)$ implies nonvanishing of
$\Ext^n_{KG}(K,K)$ for every $n\geq1$. 
This answers a conjecture by Erdmann, Kl\'asz and Marczinzik.
We use this to show that the Hochschild cohomology of $KG$ is nonzero
in every nonnegative degree whenever the characteristic of $K$
divides $|G|$.
\end{abstract}
\maketitle

\section*{Introduction}
The cohomology of finite groups with integral or field coefficients is a classical object of study in mathematics that connects various areas of algebra and topology. For a finite
group $G$ and a commutative ring $R$, the natural isomorphism
\[
 \HH^*(G,R)=\Ext^*_{RG}(R,R)\cong\HH^*(BG,R)
\]
identifies the Yoneda product on the selfextensions of the trivial
module $R$ with the cup product on the cohomology of the classifying
space $BG$, see for example \cite{Ben2}. For $R=\Z$, the groups
$\HH^n(G,\Z)$ are finite in every positive degree, so their orders
and torsion structure provide arithmetic invariants of $G$.
The integral cohomology ring also contains the Chern classes of
complex representations and is related to topological $K$-theory
through the work of Atiyah \cite{Ati}. Reduction of coefficients
modulo primes relates these integral groups to modular cohomology.

Over a field $K$ of characteristic $p>0$, finite generation
\cite{Eve} makes the graded-commutative ring $\HH^*(G,K)$ accessible
to commutative algebra and algebraic geometry. Quillen's work
\cite{Qui1,Qui2} describes its spectrum in terms of
elementary abelian $p$-subgroups; in particular, its Krull dimension
is the largest rank of such a subgroup. 

In positive characteristic, the ring structure is fundamental to
modular representation theory. Tensoring extensions of the trivial
module with a finite-dimensional $KG$-module $M$ gives a natural action of
$\HH^*(G,K)$ on $\Ext^*_{KG}(M,M)$. The resulting support theory
detects projectivity and measures the growth of minimal projective
resolutions \cite{Car,Ben2}. At the level of triangulated
categories, Benson, Iyengar and Krause \cite{BIK} classify the
tensor-ideal localising subcategories of the stable category of all
$KG$-modules by subsets of $\operatorname{Proj}\HH^*(G,K)$. These results explain
why the selfextensions of a single module, the trivial module,
provide a geometric framework for studying the entire module
category.

In \cite[Conjecture K~2.11]{EKM}, Erdmann, Kl\'asz and Marczinzik gave the following conjecture:
\begin{conjectureintro}
Let $G$ be a finite group with group algebra $KG$ over the field $K$. Then $\Ext^1_{KG}(K,K)\neq0$ implies
$\Ext^n_{KG}(K,K)\neq0$ for all $n\geq1$.
\end{conjectureintro}
In \cite{EKM} this conjecture and several extensions were proven for special cases such as when $KG$ is representation-finite or of tame representation type.

We will prove the previous conjecture in this article as a corollary of our main result, which gives the following surprising result for integral group (co)homology:
\begin{thmintroduction}\label{thm:main}
Let $G$ be a finite group. Then, for every $m\geq1$,
\[
 |G/G'|\ \bigm|\ |\HH^{2m}(G,\Z)|
       =|\HH_{2m-1}(G,\Z)|.
\]
\end{thmintroduction}

Here $\HH_i(G,\Z)$ denotes integral group homology and $G'$ denotes the commutator subgroup of $G$. Since
$\HH_1(G,\Z)\cong G/G'$, the equality of orders in
Theorem~\ref{thm:main}, applied with $m=1$, gives
$|\HH^2(G,\Z)|=|\HH_1(G,\Z)|=|G/G'|$. Thus the theorem
can also be expressed as the two divisibilities
\[
 |\HH^2(G,\Z)|\mid|\HH^{2m}(G,\Z)|,
 \qquad
 |\HH_1(G,\Z)|\mid|\HH_{2m-1}(G,\Z)|
 \qquad(m\geq1).
\]

The proof of Theorem~\ref{thm:main} combines arguments from group theory and character theory with methods from homological algebra and linear algebra over the integers.

We now turn to group algebras over a field. We may assume that
the characteristic is positive, since group algebras in
characteristic zero are semisimple and their positive-degree
cohomology vanishes.

\begin{corintroduction}\label{cor:modules}
Let $G$ be a finite group and let $K$ be a field of characteristic
$p>0$ such that $\Ext^1_{KG}(K,K)\neq0$. Then the following hold.
\begin{enumerate}
\item $\Ext^n_{KG}(K,K)\neq0$ for every $n\geq1$.
\item If $M$ is a finite-dimensional $KG$-module and
$p\nmid\dim_K M$, then $\Ext^n_{KG}(M,M)\neq0$ for every $n\geq1$.
\end{enumerate}
\end{corintroduction}

The first assertion is a proof of the above 
conjecture due to Erdmann, Kl\'asz and Marczinzik. The second assertion makes the consequence in \cite[Proposition 2.12]{EKM} unconditional and works over any field of characteristic $p$.

Our main result also has new consequences for Hochschild cohomology of finite groups. For a
$K$-algebra $B$, write
$\Hoch^n(B)=\Ext^n_{B\otimes_K B^{\mathrm{op}}}(B,B)$ for the $n$-th Hochschild cohomology of $B$.
The Hochschild cohomology ring $\Hoch^*(B)$ connects representation
theory with deformation theory and the study of derived categories.
We refer to Witherspoon \cite{Wit} for a systematic treatment of
Hochschild cohomology and its applications.
Its cup product and Gerstenhaber bracket give it a Gerstenhaber
algebra structure \cite{Ger63}; its second cohomology classifies
infinitesimal associative deformations up to equivalence, while its
third cohomology contains the obstructions to extending them \cite{Ger64}.
The graded Hochschild cohomology ring is preserved by derived equivalences \cite{Ric},
so it can distinguish algebras with different derived module
categories. For a group algebra $KG$, the group cohomology ring
$\HH^*(G,K)$ embeds naturally as a graded subalgebra of
$\Hoch^*(KG)$. 
Our next result shows non-vanishing of all terms of the Hochschild cohomology of finite groups.
This result is obtained by combining Corollary~\ref{cor:modules} with the centraliser decomposition of Hochschild cohomology and the main result of \cite{FJL}.

\begin{corintroduction}\label{cor:hochschild}
Let $G$ be a finite group and let $K$ be a field of characteristic
$p>0$ such that $p\mid|G|$. Then
\[
 \Hoch^n(KG)\neq0\qquad\text{for every }n\geq0.
\]
\end{corintroduction}
We recall the necessary preliminaries in Section 1 and prove the main results in Section 2. 
For standard background on representation theory and homological
algebra, we refer for example to the textbooks  \cite{Ben,Ben2,Rot,Wei,Zim}. For background on integral and modular group cohomology, we refer
to the books \cite{AM,Ben,Ben2,Bro,CTVZ}.

\section{Preliminaries}
\newtheorem{lemma}[theorem]{Lemma}
\newtheorem{proposition}[theorem]{Proposition}
\theoremstyle{definition}
\newtheorem{definition}[theorem]{Definition}
\theoremstyle{plain}
All groups in this article are finite and we use right modules. We write $G'=[G,G]$ for the commutator subgroup of a finite group $G$ and for a finitely generated abelian group $T$, its torsion subgroup
is denoted by $\operatorname{tors}(T)$. For a commutative ring $R$
and a right $RG$-module $X$, we write
\[
 \HH^n(G,X)=\Ext^n_{RG}(R,X),
\]
where $R$ is the trivial module. Integral group homology is
$\HH_n(G,\Z)=\operatorname{Tor}^{\Z G}_n(\Z,\Z)$, with
trivial coefficients on both sides. For a chain complex
$Y_\bullet$ with differential $b_i:Y_i\to Y_{i-1}$, we use
$\HH_i(Y_\bullet)=\ker b_i/\im b_{i+1}$. Given a natural number $n$, we denote by $C_n$ the cyclic group of order $n$.

We recall the standard results needed for Theorem~\ref{thm:main},
giving references for their proofs. The next definition and lemma use the right-module version of the
homogeneous bar resolution in \cite[\S3.4, pp.~63--64]{Ben}.

\begin{definition}\label{def:bar-resolution}
For $i\geq0$, let $P_i=\Z[G^{i+1}]$, with right $G$-action
\[
 (g_0,\ldots,g_i)g=(g_0g,\ldots,g_ig).
\]
For $i\geq1$, the differential is
\[
 \partial_i(g_0,\ldots,g_i)=
 \sum_{j=0}^i(-1)^j
 (g_0,\ldots,\widehat{g_j},\ldots,g_i),
\]
where a hat denotes omission. The augmentation
$\partial_0:P_0\to\Z$ sends every $(g_0)$ to $1$.
\end{definition}

The standard-resolution construction cited above gives the
following facts over any commutative coefficient ring.

\begin{lemma}\label{lem:bar-resolution}
The augmented complex $P_\bullet\to\Z$ is a free right
$\Z G$-resolution. The tuples ending in $1$ form a free basis,
so $P_i\cong(\Z G)^{r_i}$ with $r_i=|G|^i$.
For every commutative ring $R$, the augmented complex
$P_\bullet\otimes_\Z R\to R$ is a free right $RG$-resolution
of the trivial module.
\end{lemma}

Unless an augmentation is explicitly mentioned, $P_\bullet$
means the unaugmented complex, with degree-zero differential zero
and degree-zero homology the trivial module.

\begin{definition}\label{def:coinvariants}
For a right $RG$-module $V$, its module of coinvariants is
\[
 V_G=V/\langle vg-v:v\in V,\ g\in G\rangle_R
 \cong V\otimes_{RG}R,
\]
where $R$ is the trivial left module.
\end{definition}
For the tensor description, apply for example the right-module version of
\cite[Proposition~2.3.8, pp.~144--145]{Lin} to the augmentation
ideal of $RG$, $I:=\langle g-1\colon g\in G\rangle_R$.

The exactness assertion in the next lemma is a consequence of Maschke's
theorem, since every short exact sequence of $FG$-modules splits;
see for example \cite[Corollary~3.6.12, p.~72]{Ben}.

\begin{lemma}\label{lem:averaging}
If $F$ is a field in which $|G|$ is invertible, the coinvariant
functor $V\mapsto V_G$ on right $FG$-modules is exact.
\end{lemma}

For a complex $X_\bullet$ of finite free right modules over
a ring $R$, choose a basis in each degree. By the
differential matrix in degree $i$, we mean the matrix
over $R$ representing the differential
$d_i \colon X_i \to X_{i-1}$ with respect to the chosen
bases of its source and target.

We will also need the following proposition.

\begin{proposition}\label{prop:normal-coinvariants}
Let $N\trianglelefteq G$, put $Q=G/N$ and $A=\Z Q$, and set
\[
 C_\bullet=P_\bullet\otimes_{\Z N}\Z,
 \qquad Y_\bullet=P_\bullet\otimes_{\Z G}\Z.
\]
Then $C_\bullet$ is a complex of free right $A$-modules, with
$Q$-action $[x](gN)=[xg]$, and
\[
 C_i\cong A^{r_i},\qquad
 C_\bullet\otimes_A\Z\cong Y_\bullet,\qquad
 Y_i\cong\Z^{r_i},\qquad r_i=|G|^i.
\]
Here $\Z$ is an $A$-module through the augmentation
\[
 \varepsilon:A\to\Z,\qquad
 \varepsilon\!\left(\sum_{u\in Q}a_u u\right)=\sum_{u\in Q}a_u.
\]
If $D_i$ denotes the differential matrix of $C_\bullet$ in
chosen $A$-bases, then the differential matrix of $Y_\bullet$
in the induced bases is $B_i=\varepsilon(D_i)$, with
$\varepsilon$ applied entry by entry. Moreover,
\[
 \HH_i(Y_\bullet)=\HH_i(G,\Z),\qquad
 \HH_i(C_\bullet\otimes_\Z\mathbb C)\cong
 \begin{cases}
 \mathbb C\text{ with trivial }Q\text{-action},&i=0,\\
 0,&i>0.
 \end{cases}
\]
\end{proposition}

\begin{proof}
The quotient-group action and the identification
$(\Z G)_N\cong\Z Q$ follow from
\cite[Proposition~1.6.4, p.~41]{Lin}. Together with
\cite[Proposition~2.3.8, pp.~144--145]{Lin} and associativity
of tensor products, this gives the displayed descriptions of
$C_\bullet$ and $Y_\bullet$. Tensoring with the trivial
$A$-module applies $\varepsilon$ to each matrix entry.
The first homology identity is the definition of group homology
using the bar resolution. For the second, apply the exact
$N$-coinvariant functor of Lemma~\ref{lem:averaging} to the
complexified augmented resolution, and then omit its augmentation.
\end{proof}

\begin{definition}\label{def:characters}
For a finite abelian group $Q$ of order $q$, write
$\widehat Q=\Hom(Q,\mathbb C^\times)$ and denote the trivial
character by $1$. Each character extends to an algebra homomorphism
\[
 \chi:\mathbb C Q \to\mathbb C,\qquad
 \chi\!\left(\sum_{u\in Q}a_u u\right)=\sum_{u\in Q}a_u\chi(u).
\]
Evaluation at the trivial character is the augmentation.
\end{definition}

The next proposition is the abelian case of
\cite[Theorem~3.3.1 and Corollaries~3.3.2--3.3.3,
pp.~262--263]{Lin}.
\begin{proposition}\label{prop:fourier}
There are $q:=|Q|$ characters in $\widehat Q$, and evaluation gives
an algebra isomorphism
\[
 \mathbb C Q \longrightarrow\prod_{\chi\in\widehat Q}\mathbb C,
 \qquad a\longmapsto(\chi(a))_\chi.
\]
For $a=\sum_{u\in Q}a_u u$, its coefficients are recovered by
\begin{equation}\label{eq:arithmetic-fourier}
 a_u=\frac1q\sum_{\chi\in\widehat Q}\chi(a)\chi(u)^{-1}.
\end{equation}
The elements
\[
 e_\chi=\frac1q\sum_{u\in Q}\chi(u)^{-1}u
\]
are pairwise orthogonal idempotents with sum $1$, and
$ae_\chi=\chi(a)e_\chi$. Every right $\mathbb C Q $-module
$V$ decomposes as $V=\bigoplus_{\chi\in\widehat Q}Ve_\chi$.
\end{proposition}
For a right $\mathbb{C} Q$-module $V$, we call
\[
Ve_\chi
=
\{v \in V : vu = \chi(u)v \text{ for every } u \in Q\}
\]
the $\chi$-character component of $V$. Thus it is the
subspace on which each $u \in Q$ acts as multiplication by
$\chi(u)$. In particular, the trivial character component
$Ve_1$ is the subspace of $Q$-invariant vectors. For a complex
of $\mathbb{C} Q$-modules, taking this component in each
degree gives a subcomplex, called its
$\chi$-character component. Note that $\HH_n(X_\bullet e_\chi)\cong \HH_n(X_\bullet)e_\chi$.

The following lemma applies this decomposition to our complexes.

\begin{lemma}\label{lem:character-complex}
Suppose that $Q=G/N$ in Proposition~\ref{prop:normal-coinvariants}
is abelian. The $\chi$-character component
$(C_\bullet \otimes_{\mathbb{Z}} \mathbb{C})e_\chi$
has terms $\mathbb{C}^{r_i}$ and differential matrices
$\chi(D_i)$, where $\chi$ is applied entrywise. Its positive-degree
homology is zero, and its degree-zero homology has dimension
$\delta_{\chi,1}$, where $\delta_{\chi,1}=1$ if $\chi=1$
and $\delta_{\chi,1}=0$ otherwise.
\end{lemma}

\begin{proof}
Multiplication by $e_\chi$ is an exact projection onto a direct
summand and hence commutes with homology. The assertions follow
from Proposition~\ref{prop:normal-coinvariants} and the identity
$ae_\chi=\chi(a)e_\chi$; on the trivial module, $e_1$ acts
as the identity and all other $e_\chi$ act as zero.
\end{proof}

We will use the standard fact that subgroups of finite free
abelian groups are finite free
\cite[Corollary~4.15(i), p.~163]{Rot}. We also use the Smith normal
form and its description by minors; see for example
\cite[Theorems~2.1(4) and~2.4]{Sta}.

\begin{proposition}\label{prop:smith-minors}
Let $B:\Z^b\to\Z^a$ have rank $t>0$ over $\Q$.
There are integral changes of bases making $B$ diagonal, with
positive nonzero entries $e_1\mid e_2\mid\cdots\mid e_t$.
In particular,
\[
 \operatorname{coker}B\cong
 \Z^{a-t}\oplus\bigoplus_{j=1}^t\Z/e_j\Z,
\]
and
\[
 |\operatorname{tors}(\operatorname{coker}B)|
 =e_1\cdots e_t
 =\gcd\{\text{all }t\times t\text{ minors of }B\}.
\]
The greatest common divisor is positive; at least one of these
minors is nonzero because $B$ has rank $t$.
\end{proposition}

The finite ranks of the bar coinvariant complex imply that
$\HH_i(G,\Z)$ is finitely generated. The transfer theorem,
applied to the trivial subgroup, shows that $|G|$ annihilates
$\HH_i(G,\Z)$ for $i>0$; see
\cite[Theorem~9.94, p.~577]{Rot}. These two facts give the
standard finiteness result below.

\begin{lemma}\label{lem:finite-homology}
For every finite group $G$, the group $\HH_i(G,\Z)$ is finite
for each $i>0$.
\end{lemma}

To apply Proposition~\ref{prop:smith-minors} to a homology group,
we use the following elementary observation.

\begin{lemma}\label{lem:homology-cokernel}
Let $Y_\bullet$ be a complex of finite free abelian groups in
nonnegative degrees, with differentials $b_i$ and $b_0=0$.
For $n\geq1$, there is a canonical exact sequence
\[
 0\longrightarrow\HH_{n-1}(Y_\bullet)
 \longrightarrow\operatorname{coker}b_n
 \longrightarrow\im b_{n-1}\longrightarrow0.
\]
If $\HH_{n-1}(Y_\bullet)$ is finite,  then $\HH_{n-1}(Y_\bullet)$ is isomorphic to $\operatorname{tors}(\operatorname{coker}b_n)$.
\end{lemma}

\begin{proof}
The first statement follows from the snake lemma applied to $\im b_n\hookrightarrow \ker b_{n-1}$ and $0\rightarrow \ker b_{n-1}\rightarrow Y_{n-1}\rightarrow \im b_{n-1}\rightarrow 0$.
$\im b_{n-1}$ is free abelian, so all torsion of $\operatorname{coker} b_n$ lies in  $Z:=\im(\HH_{n-1}(Y_\bullet)
\rightarrow\operatorname{coker}b_n)$. So if $Z$ is finite, all its elements are torsion. 

\end{proof}
Let $Y_\bullet=P_\bullet\otimes_{\Z G}\Z$ as in Proposition~\ref{prop:normal-coinvariants}. Since
$\Z$ has trivial $G$-action, every $\Z G$-linear map
$P_i\to\Z$ factors uniquely through the coinvariants. Thus
there is an isomorphism of cochain complexes
\[
 \Hom_{\Z G}(P_\bullet,\Z)
 \cong\Hom_\Z(Y_\bullet,\Z),
\]
whose cohomology is $\HH^n(G,\Z)$. The terms of $Y_\bullet$
are free abelian, and so are its boundary groups, by
\cite[Corollary~4.15(i), p.~163]{Rot}. The universal coefficient
theorem \cite[Theorem~7.59(i), \mbox{pp.~451--452}]{Rot}
therefore applies and gives, for $n\geq1$,
\[
 0\longrightarrow\Ext^1_\Z(\HH_{n-1}(G,\Z),\Z)
 \longrightarrow\HH^n(G,\Z)
 \longrightarrow\Hom_\Z(\HH_n(G,\Z),\Z)
 \longrightarrow0.
\]
For $n\geq2$, Lemma~\ref{lem:finite-homology} makes the
rightmost term zero and the homology in the leftmost term finite.
Together with the standard computation of $\Ext^1_\Z(T,\Z)$
for finite abelian $T$
\cite[Example~7.23(i)--(ii), p.~420]{Rot}, this gives the
following comparison.

\begin{proposition}\label{prop:integral-comparison}
For every finite group $G$ and $n\geq2$, there is a natural
isomorphism
\[
 \HH^n(G,\Z)\cong\Ext^1_\Z(\HH_{n-1}(G,\Z),\Z).
\]
In particular, $\HH^n(G,\Z)$ and $\HH_{n-1}(G,\Z)$ are
finite and noncanonically isomorphic, and
\[
 |\HH^n(G,\Z)|=|\HH_{n-1}(G,\Z)|.
\]
\end{proposition}

We conclude this section by recalling the following standard description and nonvanishing criterion for first cohomology with trivial coefficients.
\begin{lemma}\label{remarkone} Let $G$ be a finite group and $K$ a field of characteristic $p>0$, regarded as a trivial $KG$-module.
    \begin{enumerate}
        \item Denoting by $K^+$ the additive group with underlying set $K$, the following isomorphism holds: $\HH^1(G, K)\cong \Hom(G, K^+)$. 
\item $\Hom(G, K^+)$ is nonzero if and only if $p\mid|G/G'|$. 
    \end{enumerate}
\end{lemma}
\begin{proof}
For (1), note that since the action of $G$ on $K$ is trivial, a $1$-cocycle is a map $f:G\to K$
with $f(gh)=f(g)+f(h)$, and every $1$-coboundary is zero (see for instance \cite[Definition~1.2.7 and the following paragraph,
p.~17]{Lin}).

For (2), every homomorphism $G\to K^+$ factors through $G/G'$,
since $K^+$ is abelian. The image of any nonzero such homomorphism is  a nontrivial
finite elementary abelian $p$-group, and hence $p\mid|G/G'|$. 
Conversely, if $p\mid |G/G'|$, then the finite abelian group $G/G'$ has a quotient isomorphic to $C_p$, which embeds into $K^+$. Composing these maps yields a nonzero homomorphism $G\rightarrow K^+$.
\end{proof}
\section{Proofs of the theorem and its corollaries}
We begin with an elementary dimension calculation for homology that will be used in the main proof.

\begin{lemma}\label{lem:homology-dimension}
Let $F$ be a field and let $X_\bullet$ be a complex of
finite-dimensional $F$-vector spaces in nonnegative degrees,
with differentials $b_n:X_n\to X_{n-1}$ and $b_0:X_0\to0$.
Put $r_n=\dim_F X_n$ and $s_n=\rank_F b_n$ for $n\geq0$,
so that $s_0=0$. Then, for every $n\geq0$,
\begin{equation}\label{eq:homology-dimension}
 \dim_F\HH_n(X_\bullet)=r_n-s_n-s_{n+1}.
\end{equation}
In particular, if $\HH_n(X_\bullet)=0$ for every $n>0$ and
$h=\dim_F\HH_0(X_\bullet)$, then
\[
 s_1=r_0-h,\qquad
 s_{n+1}=r_n-s_n\quad(n\geq1).
\]
\end{lemma}
\begin{proof}
Since $\im b_{n+1}\subseteq\ker b_n$, rank--nullity gives
\[
 \dim_F\HH_n(X_\bullet)
 =\dim_F\ker b_n-\dim_F\im b_{n+1}
 =r_n-s_n-s_{n+1}.
\]
Taking $n=0$ and using $s_0=0$ gives $s_1=r_0-h$.
In positive degrees, vanishing of homology gives
$s_{n+1}=r_n-s_n$.
\end{proof}

We have now all the tools necessary to prove Theorem~\ref{thm:main}.

\begin{proof}[Proof of Theorem~\ref{thm:main}]
By Proposition~\ref{prop:integral-comparison},
$|\HH^{2m}(G,\Z)|=|\HH_{2m-1}(G,\Z)|$ for every $m\geq1$.
It therefore suffices to prove that

$$|G/G'|\mid |\HH_{2m-1}(G,\mathbb Z)|
\qquad\text{for every }m\ge 1.
$$
Put $Q=G/G'$, $q=|Q|$, and $A=\Z Q$. If $q=1$, the
divisibility is immediate, so assume $q>1$.
The key point will be that, in even degrees, the differential
on the trivial character component has rank one greater than
on each nontrivial component. This will force the relevant
integral minors to be divisible by $q$.
Apply Proposition~\ref{prop:normal-coinvariants} with $N=G'$.
It gives the complexes
\[
 C_\bullet=P_\bullet\otimes_{\Z G'}\Z,
 \qquad
 Y_\bullet=C_\bullet\otimes_A\Z
 \cong P_\bullet\otimes_{\Z G}\Z,
\]
with $C_i\cong A^{r_i}$ and $Y_i\cong\Z^{r_i}$, where
$r_i=|G|^i$. Choose free $A$-bases of the $C_i$, write $D_i$
for the differential matrices, and put $B_i=\varepsilon(D_i)$, with $\varepsilon: A\rightarrow \Z$ the augmentation map.
Thus $B_i$ is the differential matrix of $Y_\bullet$,
$B_0=0$, and $\HH_i(Y_\bullet)=\HH_i(G,\Z)$.
We identify these matrices with the homomorphisms they represent.
For $\chi\in\widehat Q$, set
$s_i(\chi)=\rank_{\mathbb C}\chi(D_i)$ for $i\geq1$ and
$s_0(\chi)=0$. Apply Lemma~\ref{lem:homology-dimension} to
\[
 X_\bullet=(C_\bullet\otimes_\Z\mathbb C)e_\chi.
\]
By Lemma~\ref{lem:character-complex}, its terms have dimensions
$r_i$, its differentials have matrices $\chi(D_i)$, and its
homology is zero in positive degrees and has dimension
$\delta_{\chi,1}$ in degree zero. Thus
\eqref{eq:homology-dimension} reads
\[
 \dim_{\mathbb C}\HH_n(X_\bullet)
 =r_n-s_n(\chi)-s_{n+1}(\chi)\qquad(n\geq0).
\]
In degree zero, this gives
\[
 s_1(\chi)=r_0-\delta_{\chi,1},
\]
and in degree $i-1>0$, it gives
\[
 s_i(\chi)=r_{i-1}-s_{i-1}(\chi)
\]
For any nontrivial $\chi$, the first rank difference is
$s_1(1)-s_1(\chi)=(r_0-1)-r_0=-1$.
Since $r_{i-1}$ is independent of $\chi$, subtracting the
two recurrences gives
\[
 s_i(1)-s_i(\chi)
 =-\bigl(s_{i-1}(1)-s_{i-1}(\chi)\bigr)\qquad(i\geq2).
\]
Induction therefore yields
\[
 s_i(1)-s_i(\chi)=(-1)^i\qquad(i\geq1,\ \chi\neq1).
\]

Fix \(m\geq1\) and put \(t=s_{2m}(1)\). By the preceding identity,
\[
s_{2m}(\chi)=t-1
\qquad\text{for every nontrivial character }\chi.
\]
Since \(Q\) is abelian and \(q>1\), such a character exists. As \(s_{2m}(\chi)\) is a matrix rank, it is nonnegative, and hence \(t\geq1\).
Evaluation at the trivial character
is augmentation, so
\[
 \rank_{\Q}B_{2m}=\rank_{\mathbb C}B_{2m}=t,
 \qquad
 \rank_{\mathbb C}\chi(D_{2m})=t-1\quad(\chi\neq1).
\]
The first two ranks agree because the minors of $B_{2m}$ are
integers, and a minor is nonzero over $\Q$ exactly when it is
nonzero over $\mathbb C$.

The ring $A$ is commutative because $Q$ is abelian, so
determinants over $A$ are defined. Take any $t\times t$ minor
$a\in A$ of $D_{2m}$. Since determinants are preserved under homomorphisms of commutative rings, evaluation at $\chi$ commutes with taking minors. The preceding rank computation therefore shows
that $\chi(a)=0$ for every $\chi\neq1$.
Writing $a=\sum_{u\in Q}a_u u$, Fourier inversion
\eqref{eq:arithmetic-fourier} gives
\[
 a_u=\frac{\varepsilon(a)}q\qquad(u\in Q).
\]
All these coefficients are integers. Thus there is an integer
$z$ such that $a=z\sum_{u\in Q}u$, and
$\varepsilon(a)=zq$ is divisible by $q$.

Augmentation also commutes with determinants, so each
$t\times t$ minor of $B_{2m}$ is the augmentation of the
corresponding minor of $D_{2m}$. Every such minor is therefore
divisible by $q$. At least one is nonzero, since $B_{2m}$ has
rank $t$.
Proposition~\ref{prop:smith-minors} identifies the positive
greatest common divisor of these minors with the order of
$\operatorname{tors}(\operatorname{coker}B_{2m})$. Hence
\[
 q\mid\bigl|\operatorname{tors}
                   (\operatorname{coker}B_{2m})\bigr|.
\]
The group $\HH_{2m-1}(G,\Z)$ is finite by
Lemma~\ref{lem:finite-homology}. Applying
Lemma~\ref{lem:homology-cokernel} to $Y_\bullet$ therefore gives
\[
 \operatorname{tors}(\operatorname{coker}B_{2m})
 \cong\HH_{2m-1}(G,\Z).
\]
It follows that $q\mid|\HH_{2m-1}(G,\Z)|$, as required.
Together with the homology--cohomology comparison at the start
of the proof, this proves the theorem.
\end{proof}

We can now prove our first corollary:

\begin{proof}[Proof of Corollary~\ref{cor:modules}]
By assumption, $0\neq \Ext_{KG}^1(K, K)=\HH^1(G, K)$. By Lemma~\ref{remarkone}, $\Hom(G, K^+)\neq 0$ and so $p\mid |G/G'|$.
Theorem~\ref{thm:main} therefore implies that
$p\mid|\HH^{2m}(G,\Z)|$ for every $m\geq1$.
Recall that for every finite abelian group $T$ whose order is divisible by $p$,
both $T/pT$ and $T[p]=\{x\in T:px=0\}$ are nonzero. Indeed since $T$ is finite and $p\mid |T|$, Cauchy's theorem gives $T[p]\neq 0$. Moreover, multiplication by $p$ on $T$ is thus neither injective nor surjective since $T$ is finite. Applying this fact to $T=\HH^{2m}(G,\Z)$ we get
\begin{equation}
    \HH^{2m}(G,\Z)[p]\neq 0, \quad \HH^{2m}(G,\Z)/p\HH^{2m}(G,\Z)\neq 0, \ \forall m\geq 1. \label{eqtwo}
\end{equation}

Use the bar resolution $P_\bullet\to\Z$ from
Lemma~\ref{lem:bar-resolution}. Since each term $P_i$ is projective, the functor $\Hom_{\Z G}(P_i, -)$ is exact and so applying it to the exact sequence
$0\to\Z\xrightarrow{p}\Z\to\Fp\to0$ yields a short exact sequence
\[0\rightarrow \Hom_{\Z G}(P_\bullet, \Z)\xrightarrow{\cdot p} \Hom_{\Z G}(P_\bullet, \Z) \rightarrow \Hom_{\Z G}(P_\bullet, \Fp)\rightarrow 0 \]
of cochain complexes. By Lemma~\ref{lem:bar-resolution},
$P_\bullet\otimes_\Z\Fp\to\Fp$ is a free $\Fp G$-resolution,
and tensor--Hom adjunction identifies
\[
 \Hom_{\Z G}(P_\bullet,\Fp)
 \cong\Hom_{\Fp G}(P_\bullet\otimes_\Z\Fp,\Fp).
\]
Its cohomology is therefore
$\HH^j(G,\Fp)=\Ext^j_{\Fp G}(\Fp,\Fp)$.
The long exact sequence in cohomology (see for example \cite[Proposition~2.5.3(ii), p.~36]{Ben})
now gives
\[
 \HH^j(G,\Z)\xrightarrow{\cdot p}\HH^j(G,\Z)
 \longrightarrow\HH^j(G,\Fp)
 \longrightarrow\HH^{j+1}(G,\Z)
 \xrightarrow{\cdot p}\HH^{j+1}(G,\Z), \quad j\geq 1.
\]
Taking the cokernel of the first
map and the kernel of the last map gives, for $j\geq1$,
\begin{equation}\label{eq:arithmetic-modp}
 0\longrightarrow\HH^j(G,\Z)/p\HH^j(G,\Z)
 \longrightarrow\HH^j(G,\Fp)
 \longrightarrow\HH^{j+1}(G,\Z)[p]\longrightarrow0.
\end{equation}
By \eqref{eqtwo}, for even $j$ the left term is nonzero, and for odd $j$ the right
term is nonzero. Hence $\HH^j(G,\Fp)\neq0$ for every $j\geq1$.
Since the terms of the bar resolution are finite free, extension
of scalars gives an isomorphism of cochain complexes
\[
 \Hom_{KG}(P_\bullet\otimes_\Z K,K)
 \cong K\otimes_{\Fp}
       \Hom_{\Fp G}(P_\bullet\otimes_\Z\Fp,\Fp).
\]
Lemma~\ref{lem:bar-resolution} identifies their cohomology with
group cohomology. Since $K$ is faithfully flat over $\Fp$,
tensoring with $K$ commutes with cohomology and preserves
nonzero vector spaces. Consequently,
\[
 \HH^j(G,K)\cong K\otimes_{\Fp}\HH^j(G,\Fp)\neq0
 \qquad(j\geq1).
\]
This proves (1).

To show (2), consider $\End_K(M)$ with the right $G$-action given by $(f\cdot g)(m)=f(mg^{-1})g$. Since $M$ is projective over $K$, the  tensor--Hom adjunction
\cite[Proposition~3.1.8(iii), pp.~53--54]{Ben} gives
\begin{equation}
    \Ext^n_{KG}(M,M)\cong\HH^n(G,\End_K(M)). \label{eqone}
\end{equation}

We claim that $K$ is a direct summand of $\End_K(M)$ as $KG$-module.  To show that put $d=\dim_K M$. Since $p\nmid d$, $d$ is invertible in $K$ and since the trace is invariant under conjugation, the maps 
\[
 \iota:K\longrightarrow\End_K(M),\quad a\longmapsto a\,\mathrm{id}_M,
 \qquad
 r:\End_K(M)\longrightarrow K,\quad f\longmapsto d^{-1}\tr(f)
\]
are $KG$-linear and $r\circ \iota=\id_K$. It follows that $\End_K(M)\cong K\oplus \ker r$.
It follows that $\HH^n(G,K)= \Ext_{KG}^n(K, K)$ is a direct summand of $\Ext_{KG}^n(K, \End_K(M))=\HH^n(G,\End_K(M))$. By \eqref{eqone} and part (1), we get that $\Ext^n_{KG}(M,M)\neq 0$ for all $n\geq 1$.
\end{proof}
For the next proof recall that every element $g$ in a group $G$ can be written in a unique way as $g=g_{p} g_{p'}$, where $g_p$ has order a power of $p$ and the order of $g_{p'}$ is prime to $p$ such that $g_p$ and $g_{p'}$ commute. In that case $g_p$ is called the $p$-part of $g$.
\begin{proof}[Proof of Corollary~\ref{cor:hochschild}]

Since $p\mid|G|$, the main result of \cite{FJL}, also discussed
in \cite{Mur}, yields an element $x\in G$ whose order is
divisible by $p$ and whose $p$-part is not contained in $C_G(x)'$.
Put $C=C_G(x)$ and write $x_p$ for the $p$-part of $x$.
The image $x_pC'$ is a nonidentity element of $p$-power order
in $C/C'$, so $p\mid|C/C'|$. Lemma~\ref{remarkone} therefore
gives $\HH^1(C,K)\cong\Hom(C,K^+)\neq0$ (see also \cite[Proposition~2.10]{Mur}).
Thus
Corollary~\ref{cor:modules}(1), applied to $C$, gives
\[
 \HH^n(C,K)\neq0\qquad(n\geq1).
\]
 The centraliser decomposition for Hochschild cohomology
\cite[Theorem~2.11.2]{Ben2} yields an isomorphism of $K$-vector spaces
\[
 \Hoch^n(KG)\cong
 \bigoplus_{g\in\mathcal R}\HH^n(C_G(g),K)
 \qquad(n\geq 0),
\]
where $\mathcal R$ is a set of representatives of the conjugacy
classes of $G$, chosen to contain $x$.
 The summand indexed by $x$ is $\HH^n(C,K)$, which is nonzero
for every $n\geq1$. Hence $\Hoch^n(KG)\neq0$ in every positive
degree. Finally,
$\Hoch^0(KG)=Z(KG)\neq0$, since the center contains the identity. 
\end{proof}
To conclude, we remark that the following conjecture due to  Erdmann, Kl\'asz and Marczinzik \cite[Conjecture S]{EKM} is still open:
\begin{conjecture}
Let $G$ be a finite group with group algebra $KG$ for a field $K$ and let $S$ be a simple $KG$-module. If $\Ext_{KG}^1(S,S) \neq 0$, then $\Ext_{KG}^i(S,S) \neq 0$ for all $i \geq 1$.
\end{conjecture}

\section*{Acknowledgements}
The main result was found by experimenting using GAP~\cite{GAP}. The original conjectures in \cite{EKM} also profited from experiments with the GAP package 
QPA~\cite{GS} and Magma~\cite{BCP}.

\section*{Statement on the use of AI.}
The proofs for this article were found with the help of ChatGPT Astra based on computer experiments by the authors. 
ChatGPT assisted with selected aspects of the technical development, including
the formulation of auxiliary results and the elaboration of technical details,
as well as with literature searches, consistency checks, and LaTex
preparation. All AI-generated suggestions were independently reviewed and
verified by the authors, who assume full responsibility for the mathematical
arguments, results, and final content of the paper.

\end{document}